\documentclass{amsart}
\usepackage{hyperref}
\hypersetup{bookmarksnumbered} 
\usepackage{breakurl}
\usepackage{amssymb,amsmath,amsthm}
\usepackage{graphicx}
\usepackage[subrefformat=parens,labelformat=parens]{subfig}

\input xy
\xyoption{all}

\usepackage{color}
 
\newcommand{\blue}[1]{\textcolor{blue}{#1}}

\usepackage{ulem} 

\makeatletter

\let\c@table\c@figure
\makeatother

\newtheorem{theorem}{Theorem}[section]
\newtheorem*{RubinsTheorem}{Rubin's Theorem}

\newtheorem{proposition}[theorem]{Proposition}
\newtheorem*{proposition*}{Proposition}
\newtheorem{lemma}[theorem]{Lemma}

\theoremstyle{definition}

\theoremstyle{remark}
\newtheorem{remark}[theorem]{Remark}

\newcommand{\newword}[1]{\textbf{#1}}

\DeclareMathOperator{\supp}{supp}
\DeclareMathOperator{\rsupp}{rsupp}
\DeclareMathOperator{\Orb}{Orb}

\newcommand{\F}{\mathcal{F}}
\newcommand{\R}{\mathcal{R}}
\newcommand{\B}{\mathcal{B}}

\newcommand{\cl}{\mathrm{cl}}

\title{A short proof of Rubin's theorem}

\author{James Belk}
\address{School of Mathematics and Statistics, University of Glasgow, University Place, Glasgow, G12~8QQ, Scotland.}
\email{\href{mailto:jim.belk@glasgow.ac.uk}{jim.belk@glasgow.ac.uk}}

\thanks{The first author has been partially supported by EPSRC grant EP/R032866/1 as well as the National Science Foundation under Grant No.~DMS-1854367 during the creation of this paper.}

\author{Luke Elliott}
\address{Department of Mathematics, Binghamton University, Binghamton, NY 13902, USA.}
\email{\href{mailto:lelliott@binghamton.edu}{lelliott@binghamton.edu}}
\thanks{The second author has been supported by EPSRC grant EP/R032866/1 while working on a project with Dr Collin Bleak}

\author{Francesco Matucci}
\address{Dipartimento di Matematica e Applicazioni, Universit\`{a} degli Studi di Milano--Bicocca, Milan 20125, Italy.}
\email{\href{mailto:francesco.matucci@unimib.it}{francesco.matucci@unimib.it}}
\thanks{The third author is a member of the Gruppo Nazionale per le Strutture Algebriche, Geometriche e le loro Applicazioni (GNSAGA) of the Istituto Nazionale di Alta Matematica (INdAM) and gratefully acknowledges the support of the 
Funda\c{c}\~ao para a Ci\^encia e a Tecnologia  (CEMAT-Ci\^encias FCT projects UIDB/04621/2020 and UIDP/04621/2020) and of the Universit\`a degli Studi di Milano--Bicocca
(FA project ATE-2017-0035 ``Strutture Algebriche'').
}

\date{}  

\begin{document}

\begin{abstract}In a remarkable theorem, M.~Rubin proved that if a group $G$ acts in a locally dense way on a locally compact Hausdorff space $X$ without isolated points, then the space $X$ and the action of $G$ on $X$ are unique up to $G$-equivariant homeomorphism.  Here we give a short, self-contained proof of Rubin's theorem, using equivalence classes of ultrafilters on a poset to reconstruct the points of the space~$X$.
\end{abstract}

\maketitle

In 1989, Matatyahu Rubin proved a remarkable theorem about groups acting sufficiently nicely on locally compact spaces \cite[Corollary~3.5]{Rub89}.  Specifically, a faithful action of a group $G$ on a topological space $X$ is called a \newword{Rubin action} if
\begin{enumerate}
    \item $X$ is locally compact, Hausdorff, and has no isolated points, and\smallskip
    \item For each open set $U\subseteq X$ and each point $p\in U$, the closure of the orbit of~$p$ under the group
    \[
    G_U = \{g\in G \mid \supp(g)\subseteq U\}
    \]
    contains a neighborhood of~$p$.
\end{enumerate}
Here $\supp(g)$ denotes the \newword{support} of $g$, i.e.~the set of all points $p\in X$ for which $g(p)\ne p$.  Following Brin \cite{BrinHigher}, we will refer to actions satisfying condition~(2) as \newword{locally dense}.  Note that condition (2) is equivalent to Rubin's assertion that none of the points of $U$ has nowhere dense orbit under $G_U$.  

Rubin proved that any Rubin action of a group must be essentially unique:

\begin{RubinsTheorem}
If a group $G$ has Rubin actions on two topological spaces $X$ and~$Y$, then there exists a $G$-equivariant homeomorphism $X\to Y$.
\end{RubinsTheorem}

This theorem has proven quite useful in geometric group theory, where interesting examples of Rubin actions are abundant.  For example, the standard actions of Thompson's groups $F$, $T$, and $V$ on the interval $(0,1)$, the circle, and the Cantor set $\{0,1\}^\omega$, respectively, are Rubin actions, and many other Thompson-like groups have Rubin actions on associated spaces. Rubin's theorem has been used to understand the automorphisms of Thompson-like groups~\cite{BCMNO, BieriStrebel, BrinChameleon, BrinGuzman, KMN, Olukoya, Elliott} and for classifying such groups up to isomorphism~\cite{BDJ, BlLa, BrinHigher, DiMP, KimKobLod, Liousse, Lodha, NekExpanding}. It is well-known that the action of Grigorchuk's group~$\mathcal{G}$ of intermediate growth on the Cantor set of ends of the infinite binary tree is also a Rubin action, and many other self-similar groups also have Rubin actions on associated Cantor spaces.  Finally, note that the action of the full group of homeomorphisms or diffeomorphisms of any manifold is a Rubin action, and indeed Rubin's theorem implies Whittaker's theorem on the reconstruction of manifolds from their homeomorphism groups~\cite{Whi} and can be used to prove the Takens--Filipkiewicz theorem on the reconstruction of manifolds from their diffeomorphism groups \cite{Fil,Tak}.  See \cite{KimKob} for a nice exposition of all of these theorems (including Rubin's theorem) and the relationships between them, as well as further applications.

Rubin actually proved many different reconstruction theorems, and both of the proofs that Rubin gave of the above theorem (in \cite{Rub89} and later in \cite{Rub}) were in the context of this more general development.
Here we give a short, self-contained version of the proof of Rubin's theorem as stated above, following the same basic outline as Rubin's second proof \cite[Theorem~3.1]{Rub}.   Starting with a Rubin action of a group $G$ on a space~$X$, our goal is to reconstruct $X$ entirely from the algebraic structure of~$G$.  That is, we wish to use $G$ itself to construct a new space $\widetilde{X}$ on which it acts, and then prove that there is a $G$-equivariant homeomorphism $X\to\widetilde{X}$.  This proof has the following steps:
\begin{enumerate}
\item In Section~1, we define a first-order relation on $G$ which we call ``algebraic disjointness'', which we show is closely related to elements of $G$ having disjoint supports in~$X$.\smallskip
\item In Section~2 we define the ``regular support'' $U$ of any element $g\in G$ to be the interior of the closure of its support in~$X$, and we show that we can use algebraic disjointness to define the subgroup $G_U$ without reference to the action of $G$ on~$X$.\smallskip
\item Finally, in Section~3 we define the poset $\R$ of all finite, nonempty intersections of regular supports of elements of~$G$.  This is isomorphic to the poset of all of the corresponding subgroups $G_U$ ($U\in\R$).  We prove that it is possible to reconstruct the points of $X$ as equivalence classes of ultrafilters on~$\R$, and show that the resulting space $\widetilde{X}$ admits a $G$-equivariant homeomorphism $\widetilde{X}\to X$. 
\end{enumerate}
We have tried to simplify Rubin's proofs as much as possible throughout.  In Section~3, our approach differs from Rubin's in that we concentrate on the poset $\R$ instead of the full Boolean algebra of regular open sets, which leads to some simplifications in the argument.

\section{Algebraic disjointness}

Given a group $G$ and elements $f,g\in G$, we say $g$ is \newword{algebraically disjoint} from $f$ if it satisfies the following mysterious condition:
\begin{quote}
For every $h\in G$ with $[f,h]\ne 1$, there exist $f_1,f_2\in C_G(g)$ so that $\bigl[f_1,[f_2,h]\bigr]$ is a nontrivial element of $C_G(g)$.
\end{quote}
Here $C_G(g)$ denotes the centralizer of $g$ in~$G$.

The idea of algebraic disjointness is that it is an entirely algebraic property of group elements which is not very different from the statement ``$f$ and $g$ have disjoint supports'', as shown in Proposition~\ref{prop:AlgebraicDisjointness} below.  The proof below was first given by Rubin~\cite[Lemma~2.17]{Rub}, though we state the proposition with slightly more general hypotheses.

For the following proposition, we say that a group $G$ of homeomorphisms of a space $X$ is \newword{locally moving} if $G_U\ne 1$ for every nonempty open set $U\subseteq X$.  Note that this condition follows easily from local density as long as $X$ is Hausdorff and has no isolated points. (The Hausdorff condition is necessary here, e.g.~if $X$ is indiscrete and $G$ is trivial.) However, being locally moving is strictly weaker than being locally dense, since it allows the space $X$ to have invariant subsets which are nowhere dense, including global fixed points.

It is an interesting fact that every locally moving group of homeomorphisms of a nonempty Hausdorff space satisfies no laws.  This follows from a result of Ab\'ert~\cite{Abert} and was first observed by Nekrashevych (see~\cite[Theorem~3.6.24]{KimKob}).


\begin{proposition}\label{prop:AlgebraicDisjointness}
Let $G$ be a locally moving group of homeomorphisms of a Hausdorff space~$X$.  Then for all $f,g\in G$:
\begin{enumerate}
    \item If\/ $\supp(f)\cap\supp(g)=\emptyset$, then $g$ is algebraically disjoint from~$f$.\smallskip
    \item If $g$ is algebraically disjoint from $f$, then\/ $\supp(f)\cap\supp(g^{12})=\emptyset$.
\end{enumerate} \end{proposition}
\begin{proof}
For (1), suppose $f$ and $g$ have disjoint supports, and suppose that $h\in G$ and $[f,h]\ne 1$.  Then $h$ is not the identity on $\supp(f)$, so since $X$ is Hausdorff we can find a nonempty open set $V\subset \supp(f)$ such that $h(V)$ is disjoint from~$V$.  Let $f_2$ be a nontrivial element of~$G_V$.  Again, since $X$ is Hausdorff there exists a nonempty open set $W\subset V$ such that $f_2(W)$ is disjoint from~$W$.  Let $f_1$ be a nontrivial element of $G_W$, and note that $f_1,f_2\in C_G(g)$ since their supports lie in $\supp(f)$.  We claim that $\bigl[f_1,[f_2,h]\bigr]$ is a nontrivial element of $C_G(g)$.

Note first that $h f_2^{-1} h^{-1}$ is supported on~$h(V)$, and hence the commutator $k=[f_2,h]=f_2\bigl(hf_2^{-1}h^{-1}\bigr)$ agrees with $f_2$ on $V$.  Then $kf_1^{-1} k^{-1}$ is supported on $k(W)=f_2(W)$, so the commutator $[f_1,k]=f_1\bigl(kf_1^{-1}k^{-1}\bigr)$ is supported on $W\cup f_2(W)\subseteq V\subset \supp(f)$, and therefore commutes with $g$.  Furthermore, $[f_1,k]$ agrees with $f_1$ on $W$, and is therefore not the identity, which proves that $g$ is algebraically disjoint from~$f$.

For statement (2), suppose that $g$ is algebraically disjoint from $f$, and suppose to the contrary that the set $U = \supp(f)\cap \supp(g^{12})$ is nonempty.  Note that $U\subseteq \supp(g^i)$ for $i=1,2,3,4$ since $1$, $2$, $3$, and $4$ are divisors of~$12$.  Since $X$ is Hausdorff, we can find a nonempty open set $V\subseteq U$ such that $f(V)$ is disjoint from $V$ and the sets $\{g^i(V)\}_{i=0}^4$ are pairwise disjoint.  Let $h$ be a nontrivial element of~$G_V$, and note that $[f,h]\ne 1$ since $f(V)$ is disjoint from~$V$.  Since $g$ is algebraically disjoint from~$f$, there exist $f_1,f_2\in C_G(g)$ so that the commutator $h'=\bigl[f_1,[f_2,h]\bigr]$ is a nontrivial element of~$C_G(g)$.

Now observe that $\supp([f_2,h])\subseteq V\cup f_2(V)$, and by the same reasoning
\[
\supp(h') \subseteq V \cup f_1(V) \cup f_2(V) \cup f_1f_2(V).
\]
Since $h'$ is nontrivial, it has at least one point $p$ in its support.  Since $g$ commutes with $h'$, all five of the points $\{g^i(p)\}_{i=0}^4$ lie in $\supp(h')$. By the pigeonhole principle, one of the four sets $V,f_1(V),f_2(V),f_1f_2(V)$ must contain two of these points, say $g^i(p),g^j(p)\in k(V)$ for some $0\leq i<j\leq 4$ and $k\in \{1,f_1,f_2,f_1f_2\}$.  But since $g^{j-i}(V)$ is disjoint from $V$ and $k$ commutes with $g$, we know that $g^{j-i}\bigl(k(V)\bigr)$ is disjoint from $k(V)$, a contradiction since $g^i(p)$ and $g^j(p)$ both lie in $k(V)$.
\end{proof}

\begin{remark}
Though we will not need this observation, it follows easily from the definition that if $g$ is algebraically disjoint from $f$ then $[f,g]=1$; for otherwise we can choose $h=g$, and the commutator $[f_1,[f_2,h]]$ can never be nontrivial.
\end{remark}

\begin{remark}
Algebraic disjointness is not equivalent to having disjoint supports.  For example, if $G$ is the symmetric group $S_n$ and $f=g=(1\;2)$, then it is easy to prove that $g$ is algebraically disjoint from $f$ for $n=2$ or $n\geq 5$.  Similarly, if $G$ is the full group of homeomorphisms of the Cantor set $\{0,1\}^\omega$, then the element of~$G$ of order two that switches the first digit of an infinite binary sequence turns out to be algebraically disjoint from itself (see the \hyperref[sec:Appendix]{Appendix}).

Algebraic disjointness is also not a symmetric relation.  For example, if $G=S_4$, $f=(1\;2)(3\;4)$, and $g=(1\;2)$, then $g$ is algebraically disjoint from~$f$, but not vice-versa.
\end{remark}

\section{Regular supports}
Given an action of a group $G$ on a space $X$, define the \newword{regular support} of an element $g\in G$, denoted $\rsupp(g)$, to be the interior of the closure of~$\supp(g)$.  The following properties of regular supports are easy to prove:
\begin{itemize}
\item The set $\rsupp(g)$ is always a \newword{regular open set} in~$X$, i.e.~an open set which is equal to the interior of its closure.\smallskip
\item We have $\supp(g)\subseteq \rsupp(g)$ for any $g\in G$.\smallskip
\item If $U\subseteq X$ is a regular open set and $g\in G_U$, then $\rsupp(g)\subseteq U$.
\end{itemize}
%
%
The following proposition (adapted from~\cite[Proposition~2.19]{Rub}) lets us construct the group $G_U$ algebraically when $U$ is the regular support of an element of~$G$.

\begin{proposition}\label{prop:ConstructGU}
Let $G$ be a locally moving group of homeomorphisms of a Hausdorff space $X$.  Let $f\in G$, let $U=\rsupp(f)$, and let
\[
S_f = \{g^{12} \mid g\in G\text{ and $g$ is algebraically disjoint from $f$}\}.
\]
Then the centralizer of $S_f$ in $G$ is precisely~$G_U$.
\end{proposition}

First we need the following lemma.

\begin{lemma}\label{lem:InfiniteExponent}
Let $G$ be a locally moving group of homeomorphisms of a Hausdorff space~$X$. Then for each nonempty open set $U\subseteq X$, the group $G_U$ has infinite exponent.
\end{lemma}
\begin{proof}
Suppose to the contrary that some $G_U$ has finite exponent.  Then we can choose a $g\in G_U$ and a point $p\in U$ for which the period $n$ of $p$ under $g$ is as large as possible.  Since $X$ is Hausdorff, there exists a neighborhood $V$ of $p$ so that the sets $\{g^i(V)\}_{i=0}^{n-1}$ are pairwise disjoint.  Note then that every point in $V$ has period $n$ under~$g$ (this being the maximum allowed period), with $g$ cyclically permuting the sets~$g^i(V)$.  Let $h$ be any nontrivial element of~$G_V$.  Then $hg$ cyclically permutes the sets $g^i(V)$, so every point in $V$ must have  period $n$ under $hg$. But $(hg)^n$ agrees with $h$ on~$V$ and is therefore not the identity on~$V$, a contradiction.
\end{proof}

\begin{proof}[Proof of Proposition~\ref{prop:ConstructGU}]
Let $h\in G$, and suppose first that $h\in G_U$.  Consider an element $g^{12}\in S_f$, where $g\in G$ is algebraically disjoint from~$f$. By Proposition~\ref{prop:AlgebraicDisjointness}, the open sets $\supp(g^{12})$ and $\supp(f)$ must be disjoint. Then $\supp(g^{12})$ is disjoint from the closure of $\supp(f)$, so $\supp(g^{12})$ is disjoint from~$U$.  Since $h\in G_U$, it follows that $h$ commutes with $g^{12}$, and therefore $h$ lies in the centralizer of~$S_f$.

Now suppose $h\notin G_U$, so $\supp(h)$ is not contained in~$U$.  Then $\supp(h)$ is not contained in the closure of $\supp(f)$, so there exists a nonempty open set $V\subseteq \supp(h)$ which is disjoint from~$\supp(f)$.  Since $X$ is Hausdorff, there exists a nonempty open set $W\subseteq V$ so that $h(W)$ is disjoint from $W$. By Lemma~\ref{lem:InfiniteExponent}, there exists a $g\in G_W$ so that $g^{12}\ne 1$.  Since $\supp(g)\cap \supp(f)=\emptyset$, Proposition~\ref{prop:AlgebraicDisjointness} tells us that $g$ is algebraically disjoint from~$f$, and hence $g^{12}\in S_f$.  But $h$ does not commute with~$g^{12}$ since $\supp(g^{12})\subseteq W$ and $h(W)$ is disjoint from $W$, and therefore $h$ is not in the centralizer of~$S_f$.  
\end{proof}

\begin{remark}
In the case where $X$ is Hausdorff and $G$ is locally moving, it follows from Proposition~\ref{prop:ConstructGU} that the relation ``$f$ and $g$ have disjoint supports'' can be determined entirely from the algebraic structure of $G$ (as a first-order sentence), improving on the algebraic disjointness relation defined in Section~1.  Specifically, observe that $\supp(f)$ and $\supp(g)$ are disjoint if and only if the regular supports $U=\rsupp(f)$ and $V=\rsupp(g)$ are disjoint. Since $G$ is locally moving, this occurs if and only if $G_U\cap G_V= 1$, which by Proposition~\ref{prop:ConstructGU} is equivalent to the (first-order) statement that the centralizers of $S_f$ and $S_g$ intersect trivially.
\end{remark}

\section{Proof of Rubin's theorem}
Given a group $G$ of homeomorphisms of a space $X$, let $\R$ be the collection of all nonempty intersections $\rsupp(g_1)\cap \cdots \cap \rsupp(g_n)$ for $g_1,\ldots,g_n\in G$.  Note that every set in $\R$ is regular open, being a finite intersection of regular open sets.  The collection $\R$ forms a poset under inclusion, and the group $G$ acts on this poset in an order-preserving way.

By Proposition~\ref{prop:ConstructGU}, if $G$ is locally moving and $X$ is Hausdorff, then we can reconstruct the subgroups $G_{\rsupp(g)}$ for $g\in G$ entirely from the algebraic structure of $G$.  Since
\[
G_{\rsupp(g_1)\,\cap\,\cdots\,\cap\, \rsupp(g_n)} = G_{\rsupp(g_1)}\cap\cdots\cap G_{\rsupp(g_n)}
\]
for $g_1,\ldots,g_n\in G$, we can also reconstruct all subgroups $G_U$ for $U\in\R$.  By the following proposition, this allows us to reconstruct the entire poset~$\R$.

\begin{proposition}
Let $G$ be a locally moving group of homeomorphisms of a space~$X$, and let $U,V\subseteq X$ be regular open sets.  Then $U\subseteq V$ if and only if $G_U\subseteq G_V$.
\end{proposition}
\begin{proof}
Clearly $G_U\subseteq G_V$ if $U\subseteq V$.  For the converse, suppose $U\not\subseteq V$.  Since $V$ is a regular open set, it follows that $U\not\subseteq \cl(V)$, so there exists a nonempty open set $W\subseteq U$ which does not intersect~$V$.  Then any nontrivial element of $G_W$ lies in $G_U$ but not $G_V$, so $G_U\not\subseteq G_V$. 
\end{proof}

Note that the action of $G$ on $\R$ is the same as the conjugation action of $G$ on the subgroups $G_U$ for $U \in \R$, so we can use the algebraic structure of $G$ to reconstruct the poset $\R$ together with the action of $G$ on~$\R$.

The final part of the proof is to use the poset $\R$ to reconstruct the space~$X$.  This is based on the following observation.

\begin{proposition}\label{prop:BasisTopology}
Suppose we are given a Rubin action of a group $G$ on a space~$X$.  Then the elements of the associated poset $\R$ are a basis for the topology on~$X$.
\end{proposition}
\begin{proof}
Let $U$ be an open set of $X$ and $p \in U$. Since $X$ is locally compact Hausdorff, there exists a  neighborhood $W$ of $p$ such that $\cl(W)$ is compact and $\cl(W)\subseteq U$.  Since $G$ is locally dense, there is at least one $g\in G_W$ that does not fix~$p$. Then $p\in \rsupp(g)\subseteq \cl(W)\subseteq U$, which proves that $U$ is a union of sets from $\R$.
\end{proof}

Our plan is to use ultrafilters on $\R$ to reconstruct the points of~$X$. \mbox{(See \cite[I.6--7]{Bourbaki}} for a general introduction to ultrafilters in topology.) First, recall that a \newword{prefilter} on a poset $(\mathcal{P},\leq)$ is a nonempty subset $\F$ of $\mathcal{P}$ that satisfies the following condition:
\begin{quote}
For all $x,y\in \mathcal{F}$ there exists a $z\in\mathcal{F}$ such that $z\leq x$ and $z\leq y$.
\end{quote}
A maximal prefilter is known as an \newword{ultrafilter} on $\mathcal{P}$.  By Zorn's lemma, every prefilter is contained in an ultrafilter. 

Now, if $Y$ is any topological space and $\B$ is a basis for the topology on $Y$, then we can regard $\B$ as a poset under inclusion. (We assume throughout that the empty set is not an element of any basis.) An ultrafilter $\F\subseteq \B$ is said to \newword{converge} to a point $p\in Y$ if every neighborhood of $p$ contains a set from~$\F$.  The following proposition lists some well-known properties of this form of convergence.

\begin{proposition}\label{prop:Ultrafilters}
Let $Y$ be a Hausdorff space and let $\B$ be a basis for the topology on~$Y$.  Then:
\begin{enumerate}
    \item Every ultrafilter on $\B$ converges to at most one point in~$Y$.\smallskip
    \item Every point in $Y$ has at least one ultrafilter on $\B$ that converges to it.\smallskip
    \item If $F\subseteq\B$ is an ultrafilter and $p\in Y$, then $\F$ converges to $p$ if and only if every element of $\mathcal{B}$ that contains $p$ lies in~$\mathcal{F}$.
\end{enumerate}
\end{proposition}
\begin{proof}
For (1), since any two elements of $\F$ must intersect, it follows easily from the Hausdorff condition that $\F$ converges to at most one point in~$Y$.

For (2), if $p$ is any point in $Y$, then the collection $\B_p$ of all elements of $\B$ that contain $p$ is a prefilter.  By Zorn's lemma this is contained in some ultrafilter~$\F$, which then converges to~$p$.

For (3), let $\B_p$ be the collection of all elements of $\B$ that contain~$p$.  If $\B_p\subseteq\F$ then clearly $\F$ converges to~$p$.  For the converse, suppose $\F$ converges to~$p$.  We claim that $\F\cup \B_p$ is a prefilter in~$\B$.  Let $U,V\in \F\cup\B_p$.  If $U,V\in \F$ or $U,V\in\B_p$ we are done, so suppose without loss of generality that $U\in \F$ and $V\in \B_p$.  Since $\F$ converges to $p$, there exists a $V'\in\F$ so that $V'\subseteq V$, and since $\F$ is a prefilter there exists a $W\in\F$ so that $W\subseteq U\cap V'$.  Then $W\subseteq U\cap V$, which proves that $\F\cup\B_p$ is a prefilter.  Since $\F$ is a maximal prefilter, it follows that $\F\cup \B_p \subseteq \F$, and therefore $\B_p\subseteq \F$.
\end{proof}

Next, we need an easy criterion for determining whether an ultrafilter converges.  This is supplied by the following proposition, which will apply to our poset $\R$ since this poset is closed under finite, nonempty intersections. Part~(1) of the following proposition is \cite[Corollary~7.2]{Bourbaki}.

\begin{proposition}\label{prop:Ultrafilters2}
Let $Y$ be a Hausdorff space, let $\B$ be a basis for the topology on~$Y$, and suppose $\mathcal{B}$ is closed under finite, nonempty intersections.  Let $\F\subseteq\B$ be an ultrafilter on~$\B$.  Then:
\begin{enumerate}
    \item $\F$ converges to a point $p\in Y$ if and only if $p\in \bigcap_{U\in\F} \cl(U)$.\smallskip
    \item If $Y$ is locally compact, then $\F$ converges to some point in $Y$ if and only if $\F$ has at least one set whose closure is compact.
\end{enumerate}
\end{proposition}
\begin{proof}
For (1), let $\mathcal{B}_p$ be the collection of all elements of $\B$ that contain $p$.  If $\F$ converges to $p$, then by Proposition~\ref{prop:Ultrafilters}(3) we know that $\B_p\subseteq\F$.  Since any two elements of $\F$ intersect, it follows that every element of $\F$ intersects every element of~$\B_p$, and therefore $p\in\cl(U)$ for all~$U\in\F$.

For the converse, suppose $p\in\bigcap_{U\in\F} \cl(U)$.  Let $\B_p'=\B_p\cup\{Y\}$, and let
\[
\F' = \{U\cap V \mid U\in \F\text{ and }V\in\B_p'\}.
\]
Note that every set in $\F'$ is nonempty since $p\in\bigcap_{U\in\F}\cl(U)$, and therefore  $\F'\subseteq \mathcal{B}$. 
We claim that $\F'$ is a prefilter.  To see this, let $U\cap V$ and $U'\cap V'$ be elements of $\F'$, where $U,U'\in\F$ and $V,V'\in\B_p'$.  Since $\F$ is a prefilter, there exists a $U''\in\F$ which is contained in $U\cap U'$.  Then $U''\cap (V\cap V')$ is an element of $\F'$ that is contained in both $U\cap V$ and $U'\cap V'$, so $\F'$ is a prefilter.  Since $\F$ is a maximal prefilter and $\F\subseteq \F'$, we conclude that $\F=\F'$, and it follows easily that $\F$ converges to~$p$.

For the last statement, suppose that $Y$ is locally compact.  If $\F$ converges to a point $p\in Y$, then since $p$ has a neighborhood $U$ whose closure is compact, any element of $\F$ that is contained in $U$ must have  compact closure.  Conversely, if $\F$ has at least one set with compact closure, then since the closures of the elements of $\F$ have the finite intersection property, the intersection $\bigcap_{U\in\mathcal{F}} \mathrm{cl}(U)$ must be nonempty, and therefore $\F$ converges to some point~$p$ by statement~(1).
\end{proof}

Given a set $U\in \R$, let
\[
\R_{\leq U} = \{V\in\R \mid V\subseteq U\}.
\]
Also, given an ultrafilter $\F\subseteq \R$ and a subgroup $H\leq G$, define the \newword{orbit} of $\F$ under $H$ to be the set
\[
\Orb(\F,H) = \{h(U) \mid h\in H\text{ and }U\in \F\}.
\]
The following proposition shows that we can reconstruct the relation ``$\F$ converges to a point in $U$'' from the algebraic structure of~$G$.

\begin{proposition}\label{prop:SecondStage}
Suppose we are given a Rubin action of a group $G$ on a space~$X$, and let $\R$ be the associated poset.  Then for each $U\in\R$ and each ultrafilter $\F\subseteq\R$, the following are equivalent:
\begin{enumerate}
    \item $\F$ converges to some point in $U$.\smallskip
    \item $\Orb(\F,G_U)$ contains\/ $\R_{\leq V}$ for some $V\in\R$ with $V\subseteq U$.
\end{enumerate}
\end{proposition}
\begin{proof}
Suppose first that $\F$ converges to some point $p\in U$.  Since $G$ is locally dense, the closure of the orbit $\Orb(p,G_U)$ of $p$ under $G_U$ contains a neighborhood~$V$ of~$p$.  Let $g$ be a nontrivial element of $G_V$, let $V'=\rsupp(g)$, and note that $V'\subseteq U$. We claim that $\Orb(\F,G_U)$ contains $\R_{\leq V'}$

Let $W\in \R_{\leq V'}$.  Since $W$ is open and $W\subseteq \cl\bigl(\Orb(p,G_U)\bigr)$, there exists an $h\in G_U$ such that $h(p)\in W$.  Then $h^{-1}(W)$ lies in $\R$ and is a neighborhood of~$p$.  Since $\F$ converges to $p$, it follows from Proposition~\ref{prop:Ultrafilters}(3) that $h^{-1}(W)\in \F$. Since $h\in G_U$, we conclude that $W\in \Orb(\F,G_U)$, 
and therefore $\Orb(\F,G_U)$ contains~$\R_{\leq V'}$.

For the converse, suppose $\Orb(\F,G_U)$ contains $\R_{\leq V}$ for some $V\in\R$ with $V\subseteq U$.  Since $X$ is locally compact and $\R$ is a basis for the topology on~$X$ by Proposition~\ref{prop:BasisTopology}, there exists a $V'\in\R_{\leq V}$ such that $\cl(V')$ is compact and is contained in~$V$.  We know that $V'\in\Orb(\F,G_U)$, so $g(V')\in \F$ for some $g\in G_U$.  But $g(V')$ is compact, so it follows from Proposition~\ref{prop:Ultrafilters2}(2) that $\F$ converges to some point~$p\in X$.  By Proposition~\ref{prop:Ultrafilters2}(1), the point $p$ lies in $\cl(g(V'))=g(\cl(V'))$, which is a subset of~$U$ since $\cl(V')\subseteq V\subseteq U$ and $g\in G_U$, and therefore $p\in U$.
\end{proof}

We can now reconstruct the space~$X$ for a Rubin action, finishing the proof of Rubin's theorem.  Given an ultrafilter $\F\subseteq\R$ and a set $U\in\R$, write $\F\searrow U$ if $\F$ converges to a point in~$U$.  By Proposition~\ref{prop:SecondStage}, we can reconstruct the relation $\searrow$ entirely from the algebraic structure of~$G$.  Define two ultrafilters $\F,\F'\subseteq\R$ to be \newword{equivalent} if
\[
\F\searrow U\qquad\Leftrightarrow\qquad \F'\searrow U
\]
for all $U\in\R$.  If $\F\subseteq\R$ is an ultrafilter, let $[\F]$ denote its equivalence class.  Then we can reconstruct $X$ as the set
\[
\widetilde{X} = \{[\F] \mid \F\subseteq\R\text{ is an ultrafilter and $\F\searrow U$ for some $U\in\R$}\}.
\]
The sets
\[
\{[\F]\in \widetilde{X} \mid \F\searrow U\}
\]
for $U\in\R$ form a basis for a topology on~$\widetilde{X}$, and the mapping $X\to \widetilde{X}$ that sends each $p\in X$ to the collection of ultrafilters in $\R$ that converge to $p$ is a homeomorphism.  (Recall from Proposition~\ref{prop:Ultrafilters}(2) that each point in $X$ has at least one ultrafilter that converges to it.)  Finally, the action of $G$ on $\R$ induces an action of $G$ on the ultrafilters in~$\R$, which in turn defines an action of $G$ on $\widetilde{X}$, and the homeomorphism $X\to\widetilde{X}$ is clearly $G$-equivariant.

\section*{Appendix: An algebraic disjointness proof}
\label{sec:Appendix}

He we record the following proposition and its proof.

\begin{proposition*}
Let $G$ be the group of all homeomorphisms of the Cantor set\/ $\{0,1\}^\omega$, and let $f\in G$ be the homeomorphism that switches the first digit of an infinite binary sequence.  Then $f$ is algebraically disjoint from itself.
\end{proposition*}
\begin{proof}
Let $h\in G$ so that $[f,h]\ne 1$. Let $S=\{1,f,h,hf,fh,h^f\}$, and let $S'=\{s_1^{-1}s_2 \mid s\in S\}$, and note that $\bigl[h^{-1},f\bigr]\in S'$.  Then we can inductively find a point $p\in \{0,1\}^\omega$ with the following properties:
\begin{enumerate}
    \item $\bigl[h^{-1},f\bigr](p)\ne p$, and\smallskip
    \item For each $s\in S'$, either $s(p)\ne p$ or $s$ is the identity in a neighborhood of~$p$.
\end{enumerate}
Let $U$ be a neighborhood of $p$ such that for each $s\in S'$, either $s$ is the identity on $U$ or $s(U)$ is disjoint from~$U$.  Let $k\in G$ be an element of order two that is supported on $U$.  Note then that $k^s$ commutes with $k$ for all $s\in S'$, so $k^{s_1^{-1}}$ commutes with $k^{s_2^{-1}}$ for all $s_1,s_2\in S$. 
 Let $f_2=kk^f$, and note that $f_2\in C_G(f)$.  We have
\[
[f,[f_2,h]] = k^f k k^{h^{-1}f} k^{fh^{-1}f} k^{fh^{-1}} k^{h^{-1}} k^f k
\]
Since $1,f,fh,h^f,h,hf\in S$, all of these conjugates of $k$ commute.  Since $k^f$ and $k$ have order two, this expression simplifies to
\[
[f,[f_2,h]] = k^{h^{-1}}k^{h^{-1}f}k^{fh^{-1}}k^{fh^{-1}f}.
\]
Note that conjugation by $f$ maps this element to itself, so $[f,[f_2,h]]\in C_G(f)$.  To prove that this element is nontrivial, observe that these four conjugates of $k$ are supported on $h(U)$, $fh(U)$, $hf(U)$, and $h^f(U)$, respectively.  Since $f$ and $f^h$ have no fixed points and lie in~$S'$, we know that $f(U)$ and $f^h(U)$ are disjoint from $U$, and it follows that $hf(U)$ and $fh(U)$ are both disjoint from $h(U)$.  Furthermore, since $\bigl[h^{-1},f\bigr](p)\ne p$ and $\bigl[h^{-1},f\bigr]\in S'$, we know that $\bigl[h^{-1},f\bigr](U)$ is disjoint from~$U$, and hence $h^f(U)$ is disjoint from $h(U)$.  Thus $k^{h^{-1}}$ is not equal to any of the other three conjugates, so $[f,[f_2,h]]$ agrees with $k^{h^{-1}}$ on $h(U)$, and therefore $[f,[f_2,h]]\ne 1$.
\end{proof}

\section*{Acknowledgments}

The authors would like to thank Collin Bleak and Matthew Brin for encouraging us to publish this manuscript. We would also like to thank 
Sang-hyun Kim and Thomas Koberda for drawing our attention to their recent book~\cite{KimKob}.
Moreover,
the authors thank Matthew Brin for pointing out a mistake in a preliminary version of this paper, as well as James Hyde, Matthew Zaremsky, and an anonymous referee for offering helpful comments and suggestions.

\bigskip
\newcommand{\arxiv}[1]{\href{https://arxiv.org/abs/#1}{\blue{arXiv:#1}}}
\newcommand{\doi}[1]{\href{https://doi.org/#1}{\blue{doi:#1}}}
\bibliographystyle{plain}

\begin{thebibliography}{10}

\smallskip
\bibitem{Abert}
M.~Ab\'ert, Group laws and free subgroups in topological groups. \textit{Bulletin of the London Mathematical Society} \textbf{37.4} (2005): 525--534. \doi{10.1112/S002460930500425X}.

\smallskip
\bibitem{BCMNO}
C.~Bleak, P.~Cameron, Y.~Maissel, A.~Navas, F.~Olukoya,
The further chameleon groups of Richard Thompson and Graham Higman: Automorphisms via dynamics for the Higman groups~$G_{n,r}$.  To appear in \textit{Memoirs of the American Mathematical Society}. \arxiv{1605.09302}.

\smallskip
\bibitem{BDJ}
C.~Bleak, C.~Donoven, and J.~Jonu\v{s}as, Some isomorphism results for Thompson-like groups~$V_n(G)$. \textit{Israel Journal of Mathematics} \textbf{222} (2017): 1--19. \doi{10.1007/s11856-017-1580-5}.

\smallskip
\bibitem{BlLa}
C.~Bleak and D.~Lanoue, A family of non-isomorphism results. \textit{Geometriae Dedicata} \textbf{146.1} (2010): 21--26. \doi{10.1007/s10711-009-9423-9}.

\smallskip
\bibitem{BieriStrebel}
R.~Bieri and R.~Strebel, \textit{On Groups of PL-Homeomorphisms of the Real Line}.
Mathematical Surveys and Monographs \textbf{215}, American Mathematical Society, Providence, RI, 2016.

\smallskip
\bibitem{Bourbaki}
N.~Bourbaki, \textit{General Topology, Chapters 1--4}.  Elements of Mathematics, Springer, 1995. \doi{10.1007/978-3-642-61701-0}.

\smallskip
\bibitem{BrinChameleon}
M.~Brin, The chameleon groups of Richard J.\ Thompson: Automorphisms and dynamics. \textit{Publications Math\'{e}matiques de l'IH\'{E}S} \textbf{84} (1996): 5--33. \doi{10.1007/BF02698834}.

\smallskip
\bibitem{BrinHigher}
M.~Brin, Higher dimensional Thompson groups. \textit{Geometriae Dedicata} \textbf{108} (2004): 163--192. \doi{10.1007/s10711-004-8122-9}.


\smallskip
\bibitem{BrinGuzman}
M.~Brin and F.~Guzm\'{a}n, Automorphisms of generalized Thompson groups.  \textit{Journal of Algebra} \textbf{203.1} (1998): 285--348. \doi{10.1006/jabr.1997.7315}.



\smallskip
\bibitem{DiMP}
W.~Dicks and C.~Mart\'{i}nez-P\'{e}rez, Isomorphisms of Brin--Higman--Thompson groups. \textit{Israel Journal of Mathematics} \textbf{199.1} (2014): 189--218. \doi{10.1007/s11856-013-0042-7}.

\smallskip
\bibitem{Elliott}
L.~Elliott, A description of \(\operatorname{Aut}(dV_n)\) and \(\operatorname{Out}(dV_n)\) using transducers. To appear in \textit{Groups, Geometry, and Dynamics}. \doi{10.4171/ggd/697}.

\smallskip
\bibitem{Fil}
R.~Filipkiewicz, Isomorphisms between diffeomorphism groups. \textit{Ergodic Theory and Dynamical Systems} \textbf{2.2} (1982): 159--171. \doi{10.1017/S0143385700001486}.

\smallskip
\bibitem{KimKob}
S.~Kim and T.~Koberda, \textit{Structure and Regularity of Group Actions on One-Manifolds}. Springer Monographs in Mathematics, Springer, 2021. \doi{10.1007/978-3-030-89006-3}.

\smallskip
\bibitem{KimKobLod}
S.~Kim, T.~Koberda, and Y.~Lodha, Chain groups of homeomorphisms of the interval. 
 \textit{Annales Scientifiques de l'\'{E}cole Normale Sup\'{e}rieure} \textbf{52.4} (2019): 797--820. \doi{10.24033/asens.2397}.

\smallskip
\bibitem{KMN}
D.~Kochloukova, C.~Mart\'{i}nez-P\'{e}rez, and B.~Nucinkis, Fixed points of finite groups acting on generalised Thompson groups. \textit{Israel Journal of Mathematics} \textbf{187} (2012): 167--192. \doi{10.1007/s11856-011-0079-4}.

\smallskip
\bibitem{Liousse}
I.~Liousse, Rotation numbers in Thompson--Stein groups and applications. \textit{Geometriae Dedicata} \textbf{131} (2008): 49--71. \doi{10.1007/s10711-007-9216-y}.

\smallskip
\bibitem{Lodha}
Y.~Lodha, Coherent actions by homeomorphisms on the real line or an interval. \textit{Israel Journal of Mathematics} \textbf{235} (2020): 183--212. \doi{10.1007/s11856-019-1954-7}.

\smallskip
\bibitem{NekExpanding}
V.~Nekrashevych, Finitely presented groups associated with expanding maps. In \textit{Geometric and Cohomological Group Theory}, London Mathematical Society Lecture Note Series \textbf{444}, Cambridge University Press, 2017, 115--171. \doi{10.1017/9781316771327.009}.

\smallskip
\bibitem{Olukoya}
F.~Olukoya, Automorphism towers of groups of homeomorphisms of Cantor space. \textit{Israel Journal of Mathematics} \textbf{244} (2021): 883--899. \doi{10.1007/s11856-021-2196-z}.

\smallskip
\bibitem{Rub89}
M.~Rubin, On the reconstruction of topological spaces from their groups of homeomorphisms. \textit{Transactions of the American Mathematical Society} \textbf{312.2} (1989): 487--538. \doi{10.1090/S0002-9947-1989-0988881-4}.

\smallskip
\bibitem{Rub}
M.~Rubin, Locally moving groups and reconstruction problems. In \textit{Ordered Groups and Infinite Permutation Groups}, Mathematics and Its Applications \textbf{354}, Springer, Boston, MA, 1996, 121--157.
\href{https://doi.org/10.1007/978-1-4613-3443-9_5}{\blue{doi:10.1007/978-1-4613-3443-9\textunderscore 5}}.

\smallskip
\bibitem{Tak}
F.~Takens, Characterization of a differentiable structure by its group of diffeomorphisms. \textit{Bulletin of the Brazilian Mathematical Society} \textbf{10.1} (1979): 17--25. \doi{10.1007/BF02588337}.

\smallskip
\bibitem{Whi}
J.~Whittaker, On isomorphic groups and homeomorphic spaces. \textit{Annals of Mathematics} (1963): 74--91. \doi{10.2307/1970503}.

\end{thebibliography}

\end{document}